\documentclass{amsart}
\usepackage{mathtools, microtype, mathrsfs, xfrac, hyperref, amssymb, enumerate, xspace}
\usepackage[utf8]{inputenc}
\usepackage{tikz-cd}
\usepackage{todonotes}

\numberwithin{equation}{section}

\numberwithin{subsection}{section}

\allowdisplaybreaks[1]

\newtheorem*{namedtheorem}{\theoremname}
\newcommand{\theoremname}{testing}

\newtheorem{theorem}{Theorem}
\newtheorem{proposition}[theorem]{Proposition}
\newtheorem{proposition-definition}[theorem]
{Proposition-Definition}
\newtheorem{corollary}[theorem]{Corollary}
\newtheorem{lemma}[theorem]{Lemma}

\theoremstyle{definition}

\theoremstyle{remark}

\renewcommand{\mathcal}{\mathscr}

 \newcommand\cB{\mathcal{B}}
 \newcommand\cD{\mathcal{D}}
 
\newcommand\cG{\mathcal{G}}

 \newcommand\cP{\mathcal{P}}

\newcommand\GG{\mathbb{G}}

 \newcommand\PP{\mathbb{P}}

 \newcommand\ZZ{\mathbb{Z}}

 \newcommand\bD{\mathbf{D}}

 \newcommand\bP{\mathbf{P}}

\newcommand\arr{\ifinner\to\else\longrightarrow\fi}

\newcommand\arrto{\ifinner\mapsto\else\longmapsto\fi}

\renewcommand\H{\operatorname{H}}

\def\displaytimes_#1{\mathrel{\mathop{\times}\limits_{#1}}}

\def\displayotimes_#1{\mathrel{\mathop{\bigotimes}\limits_{#1}}}

\newcommand\aut{\operatorname{Aut}}

\newcommand\spec{\operatorname{Spec}}

\newcommand\id{\mathrm{id}}

\newcommand{\underaut}{\mathop{\underline{\mathrm{Aut}}}\nolimits}

\newlength{\ignora}

\renewcommand{\setminus}{\smallsetminus}

\newcommand{\PGL}{\mathrm{PGL}}

\DeclareFontFamily{U}{mathx}{\hyphenchar\font45}
\DeclareFontShape{U}{mathx}{m}{n}{
      <5> <6> <7> <8> <9> <10>
      <10.95> <12> <14.4> <17.28> <20.74> <24.88>
      mathx10
      }{}
\DeclareSymbolFont{mathx}{U}{mathx}{m}{n}
\DeclareFontSubstitution{U}{mathx}{m}{n}
\DeclareMathAccent{\widecheck}{0}{mathx}{"71}
\DeclareMathAccent{\wideparen}{0}{mathx}{"75}

\renewcommand{\epsilon}{\varepsilon}

\newcommand{\on}{\operatorname}

\newcommand{\cha}{\operatorname{char}}
\newcommand{\schk}{\on{Sch}/k}
\newcommand{\sets}{\on{Sets}}
\newcommand{\br}{\on{Br}}

\title{The field of moduli of a divisor on a rational curve}

\author{Giulio Bresciani}

\address{Scuola Normale Superiore\\Piazza dei Cavalieri 7\\
56126 Pisa\\ Italy}
\email{giulio.bresciani@gmail.com}

\begin{document}

\begin{abstract}
	Let $k$ be a field with algebraic closure $\bar{k}$ and $D \subset \mathbb{P}^{1}_{\bar{k}}$ a reduced, effective divisor of degree $n \ge 3$, write $k(D)$ for the field of moduli of $D$. A. Marinatto proved that when $n$ is odd, or $n = 4$, $D$ descends to a divisor on $\mathbb{P}^{1}_{k(D)}$.
	
	We analyze completely the problem of when $D$ descends to a divisor on a smooth, projective curve of genus $0$ on $k(D)$, possibly with no rational points. In particular, we study the remaining cases $n \ge 6$ even, and we obtain conceptual proofs of Marinatto's results and of a theorem by B. Huggins about the field of moduli of hyperelliptic curves.
\end{abstract}

\maketitle


We work over a field $k$ with algebraic closure $\bar{k}$. A rational curve is a smooth, projective, geometrically connected curve of genus $0$.

\subsection*{The field of moduli} Let $D\subset \PP^{1}_{\bar{k}}$ be an effective reduced divisor. If $k$ is perfect, consider the subgroup $H\subset\operatorname{Gal}(\bar{k}/k)$ of elements $\sigma$ such that there exists an automorphism of $\PP^{1}_{\bar{k}}/\bar{k}$ mapping $D$ to $\sigma(D)$. The \emph{field of moduli} $k(D)$ of $D$ is the subfield of $\bar{k}$ of elements fixed by $H$. In our joint paper with A. Vistoli \cite{giulio-angelo-moduli} we have generalized the definition of the field of moduli to arbitrary $k$; we work in this generality.

If $\bar{k}/k'/k$ is a sub-extension and there exists a rational curve $P$ over $k'$ with a divisor $D_{0}\subset P$ such that $(P_{\bar{k}},D_{0,\bar{k}})\simeq(\PP^{1}_{\bar{k}},D)$, we say that $k'$ is a field of definition for $(\PP^{1}_{\bar{k}},D)$. The field of moduli is contained in every field of definition.

\subsection*{Marinatto's results} Let $n$ be the degree of $D$, assume $n\ge 3$. Assuming $k$ perfect of characteristic $\neq 2$, A. Marinatto \cite{marinatto} showed that, if $n$ is odd or $n=4$, then $D$ descends to a divisor of $\PP^{1}_{k(D)}$. Hence, if $n$ is odd or $n=4$, $k(D)$ is a field of definition for $(\PP^{1}_{\bar{k}},D)$.

For every even integer $n\ge 6$, Marinatto constructed examples of divisors $D\subset\PP^{1}_{\bar{k}}$ of degree $n$ which do not descend to $\PP^{1}_{k(D)}$. His examples descend to a non-trivial rational curve over $k(D)$, though, hence $k(D)$ is a field of definition for $(\PP^{1}_{\bar{k}},D)$ in these cases. 

\subsection*{Our results} Because of this, the problem of fields of definition versus fields of moduli for $n$ even and $\ge 6$ is still open. We solve it, and generalize Marinatto's results for $n$ odd and $n=4$ to an arbitrary base field. The following theorem sums up the various results we obtain.

\begin{theorem}\label{thm:main}
    Let $k$ be any field, $n\ge 3$ an integer and $D\subset \PP^{1}_{\bar{k}}$ a reduced, effective divisor of degree $n$. If at least one of the following conditions holds, then $(\PP^{1}_{\bar{k}},D)$ is defined over its field of moduli.
    \begin{itemize}
		    \item $\aut(\PP^{1},D)$ is not cyclic of degree even and prime with $\cha k$.
		    \item $n=4$.
		    \item $n=6$.
		    \item $n$ is odd.
	\end{itemize}
	In particular, the first condition implies that $(\PP^{1}_{\bar{k}},D)$ is always defined over its field of moduli if $\cha k=2$.
	
	If $n=4$ or $n$ is odd, we can say more precisely that $D$ descends to a divisor of $\PP^{1}$ over the field of moduli.
	
	On the other hand, if $\cha k\neq 2$, $n\ge 8$ is even and the $2$-torsion $\br(k)[2]$ of the Brauer group of $k$ is non-trivial, then we may choose $D$ such that the field of moduli is $k$ and $(\PP^{1}_{\bar{k}},D)$ is not defined over $k$.
\end{theorem}

For $n=4$ or $n$ odd Theorem~\ref{thm:main} is a generalization of Marinatto's theorem to arbitrary fields. For $n$ even and $\ge 6$, it is new. In order to prove all the various sub-statements of Theorem~\ref{thm:main}, we first characterize the divisors $D\subset\PP^{1}_{\bar{k}}$ such that the field of moduli is a field of definition for $(\PP^{1}_{\bar{k}},D)$, see Theorem~\ref{thm:divisors}. 

\subsection*{Discussion of proofs} Marinatto's proof for $n$ odd is based on a case-by-case analysis of the finite subgroups of $\PGL_{2}$. Our approach is more conceptual: for $n$ odd, we only use general techniques for problems about fields of moduli (not specific to divisors or rational curves) plus a parity counting. For $n$ even and $\ge 8$, we also need to make some elementary geometric arguments. If $\aut(\PP^{1},D)$ is not cyclic of degree even and prime with $\cha k$, we have a conceptual argument based on the Riemann-Hurwitz formula; from this case we get the cases $n=4,6$, too.

Marinatto also uses a theorem of Huggins \cite[Theorem 5.3]{huggins} about fields of moduli of hyperelliptic curves. Huggins' proof is based on a case-by-case analysis of the finite subgroups of $\PGL_{2}$, too. Our argument with Riemann-Hurwitz directly yields a short, conceptual proof of Huggins' result; we include it at the end of the paper. Our version is slightly more general than the original one since we do not assume the base field to be perfect.

While we use the general framework constructed in \cite{giulio-angelo-moduli}, we mention the fact that, in dimension $1$ (which is our case), the main ideas are due to P. Dèbes and M. Emsalem \cite{debes-emsalem}.

\subsection*{Acknowledgements}

This paper was born as part of my recent joint articles with A. Vistoli \cite{giulio-angelo-valuative} and \cite{giulio-angelo-moduli}. I am grateful to him for many useful discussions, as well as for pointing out Marinatto's article to me.

\section{The residual gerbe and the compression}

For the convenience of the reader, we recall some objects and constructions from our joint paper with A. Vistoli \cite{giulio-angelo-moduli}. 

Consider the functor $\cG:\schk\to\sets$ sending a scheme $S$ over $k$ to the set of projective bundles $Q\to S$ of relative dimension $1$ with an effective divisor $E\subset Q$ finite étale over $S$ such that each geometric fiber is isomorphic to $(\PP^{1}_{\bar{k}},D)$, possibly after enlarging the base field. It can be checked that $\cG$ is an algebraic stack of finite type over $k$ with finite, étale inertia.

The algebraic stack $\cG$ is a gerbe over a field $k(D)$ finite over $k$ which is, by definition, the field of moduli of $(\PP^{1}_{\bar{k}},D)$, and the base change $\cG\times_{k(D)}\bar{k}$ is isomorphic to the classifying stack $\cB_{\bar{k}}\aut(\PP^{1}_{\bar{k}},D)$ \cite[Proposition 3.10]{giulio-angelo-moduli}. It is called the \emph{residual gerbe} of $(\PP^{1}_{\bar{k}},D)$. By definition, $(\PP^{1}_{\bar{k}},D)$ is defined over $k(D)$ if and only if the residual gerbe $\cG$ has a $k(D)$-rational point. If $k$ is finite then this is always true \cite[Theorem 8.1]{diproietto-tonini-zhang}.

There is a universal projective bundle $\cP\to\cG$ over $\cG$ with a universal divisor $\cD\subset\cP$ finite étale over $\cG$. The coarse moduli space $P$ of $\cP$ is called the \emph{compression} of $(\PP^{1}_{\bar{k}},D)$, it is a rational curve over $k(D)$ whose pullback to $\bar{k}$ identifies naturally with $\PP^{1}_{\bar{k}}/\aut(\PP^{1}_{\bar{k}},D)$ \cite[\S 5]{giulio-angelo-moduli}. Let us call \emph{compressed divisor} the coarse moduli space $\bD\subset\bP$ of $\cD$, since $\cD$ is finite étale over $\cG$ then $\bD$ is finite étale over $\spec k$.

If $k$ is perfect, we can give a more down-to-earth construction of the compression: $\on{Gal}(\bar{k}/k(D))$ acts naturally on $\PP^{1}_{\bar{k}}/\aut(\PP^{1}_{\bar{k}},D)$ which thus descends to a rational curve $\bP$ over $k(D)$, and similarly $D/\aut(\PP^{1}_{\bar{k}},D)$ descends to a divisor $\bD\subset\bP$.

There is a rational map $P\dashrightarrow \cG$ whose pullback $P_{\bar{k}}=\PP^{1}_{\bar{k}}/\aut(\PP^{1},D)\dashrightarrow \cG_{\bar{k}}=\cB_{\bar{k}}\aut(\PP^{1}_{\bar{k}},D)$ to $\bar{k}$ is associated with the projection $\PP^{1}_{\bar{k}}\to\PP^{1}_{\bar{k}}/\aut(\PP^{1}_{\bar{k}},D)$, see the proof of \cite[Theorem 5.4]{giulio-angelo-moduli}. In particular, if $P\simeq \PP^{1}_{k(D)}$ then $\cG$ has a $k(D)$-rational point.

\section{Rational maps from rational curves to gerbes}

As we have seen above, rational maps $P\dashrightarrow \cG$ where $P$ is a rational curve, $\cG$ is a finite étale gerbe and such that the geometric fibers are connected of genus $0$ play a crucial role in our problem. In this section, we study such maps.

\begin{proposition}\label{prop:cyclicgerbe}
	Let $k$ be a field, $P$ a rational curve with $P(k)=\emptyset$, $\Phi$ a finite, étale gerbe with a rational map $P\dashrightarrow\Phi$. Assume that the geometric fibers of $P\dashrightarrow\Phi$ are irreducible curves of genus $0$.
	
	The gerbe $\Phi$ is abelian with cyclic band of order prime to $\operatorname{char}k$, and there exists a separable extension $k'/k$ of degree $2$ and a point $p\in P(k')$ such that $P\dashrightarrow\Phi$ restricts to a morphism $P\setminus\{p\}\to \Phi$.
\end{proposition}

Let us first sketch the proof in the particular case in which $k$ is perfect. Write $\Phi_{\bar{k}}=\cB_{\bar{k}}G$ for some finite group $G$, the base change of $P\dashrightarrow\Phi$ to $\bar{k}$ corresponds to a $G$-covering $\PP^{1}_{\bar{k}}\to\PP^{1}_{\bar{k}}$. The branch locus of the $G$-covering descends to $P$: if $P(k)=\emptyset$, this forces every branch point $p$ of $\PP^{1}_{\bar{k}}$ to have at least one Galois conjugate $\bar{p}$ such that the ramification data over $p$ and $\bar{p}$ are equal. Using the Riemann-Hurwitz formula, it can be easily checked that this only happens for $G$ cyclic. 

\begin{proof}
	Let $\bar{k}$ be an algebraic closure and $\bar{k}\supseteq k^{s}\supseteq k$ be the separable closure, we have $P_{k^{s}}\simeq\PP^{1}_{k^{s}}$, in particular $\Phi_{k^{s}}$ is neutral since there is a rational map $P\dashrightarrow\Phi$. Choose a section $\spec k^{s}\to \Phi$ and let $G$ be its automorphism group, it is a finite constant group and there is an induced branched $G$-covering $f:\PP^{1}_{k^{s}}\to P_{k^{s}}\simeq \PP^{1}_{k^{s}}$.
	
	Let $R$ be the ramification divisor of $f$ as in \cite[\href{https://stacks.math.columbia.edu/tag/0C1B}{Tag 0C1B}]{stacks-project}, we think of $R$ as a non-reduced, finite scheme over $k^{s}$. There is an action of $G$ on $R$, in particular on the underlying set. Given a point $r\in R$ denote by $d_{r}$ the length of $R$ at $r$, $e_{r}$ the ramification index, $o_{r}$ the cardinality of the set-theoretic $G$-orbit of $r$. Since $f$ has degree $n$, then $n=o_{r}e_{r}[k^{s}(r):k^{s}(f(r))]$. By Riemann-Hurwitz \cite[\href{https://stacks.math.columbia.edu/tag/0C1F}{Lemma 0C1F}]{stacks-project} we have
	\[2n-2=\sum_{r\in R}d_{r}[k^{s}(r):k^{s}]\]
	with $d_{r}=e_{r}-1$ if $f$ is tamely ramified at $r$ and $d_{r}\ge e_{r}$ otherwise.
	
	There exists a largest open subset $U\subset P$ such that $U\to\Phi$ is defined \cite[Corollary A.2]{giulio-ed}. Write $Z=P\setminus U$, then $Z_{k^{s}}$ is the branch locus of $f$ \cite[Lemma A.1, Corollary A.2.iv]{giulio-ed}, in particular $f(R)=Z_{k^{s}}$ set-theoretically. Since $P$ is a non-trivial Brauer-Severi variety, $Z$ has no rational points. This implies that, if $r\in R$ is a point such that $k^{s}(f(r))=k^{s}$, there exists a point $r'\in R$ such that $f(r')\neq f(r)$ is a Galois conjugate of $f(r)$, and hence $o_{r}=o_{r'}$, $d_{r}=d_{r'}$, $e_{r}=e_{r'}$, $k^{s}(r)\simeq k^{s}(r')$.
	
	Let us show first that $f$ is tamely ramified at every point. Assume that $f$ has wild ramification at some point $r$, in particular $d_{r}\ge e_{r}$. If $k^{s}(f(r))=k^{s}$, by what we have said above we have
	\[2n-2\ge 2o_{r}e_{r}[k^{s}(r):k^{s}]=2n,\]
	which is absurd. If $k^{s}(f(r))\neq k^{s}$, then
	\[2n-2\ge o_{r}e_{r}[k^{s}(r):k^{s}]\ge 2o_{r}e_{r}[k^{s}(r):k^{s}(f(r))]=2n,\]
	hence we get a contradiction in this case, too.
	
	We may thus assume that $f$ has tame ramification: in particular, $k^{s}(r)=k^{s}(f(r))$ for every $r$, $n=o_{r}e_{r}$ and $d_{r}=e_{r}-1$ for every $r\in R$. Since $f$ is ramified at every $r\in R$, then $e_{r}\ge 2$ and $o_{r}\le n/2$. For every $z\in Z_{k^{s}}$, write $e_{z}$, $o_{z}$ for $e_{r}$, $o_{r}$, where $r\in R$ is some point with $f(r)=z$. We have
	\[2n-2=\sum_{r\in R}(e_{r}-1)[k^{s}(r):k^{s}]=\sum_{z\in Z_{k^{s}}}o_{z}(e_{z}-1)[k^{s}(z):k^{s}]=\]
	\[=\sum_{z\in Z_{k^{s}}}(n-o_{z})[k^{s}(z):k^{s}]\ge \deg Z\cdot n/2\]
	which implies $\deg Z\le 3$. Since $P$ is a non-trivial Brauer-Severi variety of dimension $1$, then $\deg Z$ is even, hence we may assume $\deg Z=2$ (if $\deg Z=0$, then $G$ is trivial).
	
	Assume first $Z_{k^{s}}$ contains only one point $z$ with $[k^{s}(z): k^{s}]=2$, in particular $\operatorname{char}k=2$. Then
	\[2n-2=o_{z}(e_{z}-1)[k^{s}(z):k^{s}]=2n-2o_{z},\]
	hence $o_{z}=1$ and $n=e_{z}$ is prime with $\cha k=2$, i.e. it is odd. Since $n$ is odd, the base change $f_{\bar{k}}$ of $f$ to $\bar{k}$ is tamely ramified and since $f$ has only one point then the same is true for $f_{\bar{k}}$, and the ramification index must be $n$. By Riemann-Hurwitz applied to $f_{\bar{k}}$, we have $2n-2=n-1$, i.e. $n=1$ which is in contradiction with $\deg Z=2$.
	
	If $Z_{k^{s}}$ contains two $k^{s}$-rational points, the Galois action swaps them because $Z$ has no rational points, hence
	\[2n-2=2o_{z}(e_{z}-1)=2n-2o_{z},\]
	which implies $o_{z}=1$ and $n=e_{z}$ is prime to $\operatorname{char}k$. If $r\in R$ is one of the ramification points, since $|G|=n=e_{z}$ then $r$ is fixed by $G$. Since $|G|=n$ is prime with $\cha k$, then $G$ is linearly reductive: this implies that $G$ acts faithfully on the tangent space of $r$ and is thus cyclic.
\end{proof}

\begin{corollary}\label{cor:cyclic}
	Let $k$ be a field and $D\subset \PP^{1}_{\bar{k}}$ an effective, reduced divisor of degree $n\ge 3$ defined over $\bar{k}$. If the compression of $(\PP^{1}_{\bar{k}},D)$ is not isomorphic to $\PP^{1}$ over the field of moduli, then $\aut(\PP^{1}_{\bar{k}},D)$ is cyclic of order prime to the characteristic.
\end{corollary}

\begin{lemma}\label{lem:BSgerbe}
	Let $P\dashrightarrow \Phi$ be as in Proposition~\ref{prop:cyclicgerbe}. Then $\Phi$ is neutral if and only if it has odd degree.
\end{lemma}

\begin{proof}
	By Proposition~\ref{prop:cyclicgerbe}, $\Phi$ is abelian with cyclic band $A$ of order prime to $\cha k$. We have that $\Phi$ corresponds to a class $\phi\in\H^{2}(k,A)$ and $\Phi$ is neutral if and only if $\phi$ is trivial. Let $n$ be the degree of $\Phi$, which corresponds to the degree of $A$. The group $\H^{2}(k,A)$ is abelian and $n$-torsion.
	
	Let $k'$ be the residue field of the point $p$ given by Proposition~\ref{prop:cyclicgerbe}, then $k'/k$ is a separable extension of degree $2$ and $\phi_{k'}=0$ since $k'$ splits $P$. We have $2\phi=\operatorname{cor}_{k'/k}(\phi_{k'})=0$: if $n$ is odd, then $\phi=\frac{n+1}{2}2\phi=0$.
	
	If $\Phi$ has even degree, by Proposition~\ref{prop:cyclicgerbe} $\cha k$ is prime with $n$ and hence $\cha k\neq 2$. In particular, the $2$-adic étale fundamental group of $\GG_{m}$ over $k^{s}$ is $\ZZ_{2}$. Write $U=P\setminus\{p\}$, to prove that $\Phi$ is not neutral it is then enough to show that there exists a morphism $U\to\Gamma$ where $\Gamma$ is a non-neutral gerbe banded by $\mu_{2}$, since the fact the the degree of $\Phi$ is even gives us a factorization $U\to\Phi\to\Gamma$. Such a morphism $U\to\Gamma$ is constructed in the proof of \cite[Proposition 13.2]{borne-vistoli} using parameters $p=r=2$ and $X=U$.
\end{proof}

\section{Which divisors descend to a rational curve}

We characterize which pairs $(\PP^{1}_{\bar{k}},D)$ descend to the field of moduli. We'll then prove Theorem~\ref{thm:main} using this characterization.

\begin{theorem}\label{thm:divisors}
		Let $k$ be a field and $D\subset\PP^{1}_{\bar{k}}$ an effective, reduced divisor of degree $n\ge 3$ on $\PP^{1}_{\bar{k}}$ with $n\ge 3$. The following are equivalent.
	\begin{itemize}
			\item $(\PP^{1}_{\bar{k}},D)$ is not defined over its field of moduli.
			\item The compression has no rational points and $\aut(\PP^{1}_{\bar{k}},D)$ has even degree.
			\item The compression has no rational points and $\aut(\PP^{1}_{\bar{k}},D)$ is cyclic of degree even and prime to $\operatorname{char}k$.
	\end{itemize}
\end{theorem}

\begin{proof}
	Thanks to \cite[Proposition 4.1]{giulio-angelo-moduli}, we may assume that $k$ is infinite. Furthermore, by base change we may assume that $k$ is the field of moduli. Let $\cG,\bP$ be the residual gerbe and the compression respectively, there is a rational map $\bP\dashrightarrow\cG$ whose geometric fibers are irreducible of genus $0$.
	
	If $(\PP^{1}_{\bar{k}},D)$ is not defined over $k$, i.e. $\cG$ is not neutral, then $\bP$ is not isomorphic to $\PP^{1}$ and hence $\cG$ is abelian with cyclic band of degree even and prime with $\cha k$ by Proposition~\ref{prop:cyclicgerbe} and Lemma~\ref{lem:BSgerbe}. Since $\cG_{\bar{k}}=\cB_{\bar{k}}\aut(\PP^{1}_{\bar{k}},D)$, we obtain that the first condition implies the third, which in turn clearly implies the second. If $\bP(k)=\emptyset$ and $\aut(\PP^{1}_{\bar{k}},D)$ has even degree, then by Lemma~\ref{lem:BSgerbe} the residual gerbe $\cG$ is not neutral, hence we conclude.
\end{proof}

\begin{corollary}\label{cor:cha2}
	Let $k$ be a field of characteristic $2$ and $D\subset \PP^{1}_{\bar{k}}$ an effective, reduced divisor with $n\ge 3$. Then $(\PP^{1}_{\bar{k}},D)$ is defined over its field of moduli.
\end{corollary}

\section{Divisors of degree $6$}

\begin{proposition}\label{prop:deg6}
	Let $k$ be a field and $D\subset \PP^{1}_{\bar{k}}$ an effective, reduced divisor of degree $6$. Then $(\PP^{1}_{\bar{k}},D)$ is  defined over its field of moduli.
\end{proposition}

\begin{proof}
	Thanks to Corollary~\ref{cor:cha2}, we may assume $\cha k\neq 2$. Up to base change, we may assume that $k$ is the field of moduli, let $\bP$ be the compression and $\bD\subset\bP$ the compressed divisor, we have that $\bD$ is finite étale over $k$. Assume by contradiction that $(\PP^{1}_{\bar{k}},D)$ is not defined over $k$, then by Theorem~\ref{thm:divisors} $\aut(\PP^{1}_{\bar{k}},D)$ is cyclic of degree even and prime to $\cha k$, and $\bP$ is a non-trivial rational curve.
	
	Since $\bP$ is non-trivial, then $\bD$ has even degree. Since $\cha k\neq 2$, then $\bD_{\bar{k}}=D/\aut(\PP^{1},D)$ has either $2,4$ or $6$ points. Since $\aut(\PP^{1}_{\bar{k}},D)$ is cyclic of degree even and prime with $\cha k$, then its action on $\PP^{1}_{\bar{k}}$ has two fixed points, while all the other orbits have even cardinality equal to $|\aut(\PP^{1}_{\bar{k}},D)|$. It then follows that $D$ contains the $2$ fixed points, $\bD_{\bar{k}}$ has $4$ points and $\aut(\PP^{1}_{\bar{k}},D)$ has order $2$.
	
	Choose coordinates on $\PP^{1}_{\bar{k}}$ such that $0,\infty\in D$ are the two fixed points, and $1$ is another point of $D$. With respect to these coordinates, we have that the only non-trivial element of $\aut(\PP^{1}_{\bar{k}},D)$ is the map $x\mapsto -x$, we may thus write $D=\{0,\infty,1,-1,\lambda,-\lambda\}$. It is immediate to check that $x\mapsto \lambda /x$ defines another non-trivial automorphism of $(\PP^{1},D)$, hence we get a contradiction.
\end{proof}

\section{Divisors of even degree $n\ge 8$}\label{sec:n8}

In order to construct divisors of even degree $n\ge 8$ not defined over their field of moduli, we first construct a divisor over a non-trivial rational curve with certain special properties.

\begin{lemma}\label{lem:prep}
	Let $k$ be an infinite field of characteristic $\neq 2$, $P$ a non-trivial rational curve, $n\ge 8$ an even integer.
	
	There exists a reduced, effective divisor $E\subset P$, a quadratic extension $k'/k$, a point $p\in P(k')$ with Galois conjugate $\bar{p}\in P(k')$ and a cyclic cover $f:\PP^{1}_{\bar{k}}\to P_{\bar{k}}$ of degree $2$ ramified over $p,\bar{p}$ such that the divisor $D=f^{-1}(E_{\bar{k}})$ with the reduced structure has degree $n$ and $\aut(f)$ is its own centralizer in $\aut(\PP^{1}_{\bar{k}},D)$.
\end{lemma}

\begin{proof}
	Let $\omega$ be the canonical bundle of $P$. Write either $n=4m$ or $n=4m-2$ with $m\ge 2$, depending on the class of $n$ modulo $4$. Consider the natural map $\PP(\H^{0}(\omega^{-1}))\times\PP(\H^{0}(\omega^{-m+1}))\to\PP(\H^{0}(\omega^{-m}))$ corresponding to the sum of divisors, it is dominant.
	
	If $n\ge 10$, i.e. $m\ge 3$, a generic rational point of $\PP(\H^{0}(\omega^{-m}))$ corresponds to a divisor étale over $k$ with trivial automorphism group scheme. Moreover, any rational point of $\PP(\H^{0}(\omega^{-1}))$ corresponds to a point of $P$ whose residue field is a quadratic extension of $k$. Taking a generic rational point of $\PP(\H^{0}(\omega^{-1}))\times\PP(\H^{0}(\omega^{-m+1}))$, we can thus find a divisor $E\subset P$ of degree $2m$ étale over $k$ with trivial automorphism group scheme and a quadratic extension $k'/k$ such that $E(k')\neq\emptyset$. If $n=8$, with an analogous argument we get a reduced divisor $E\subset P$ of degree $4$ such that $\aut(\PP^{1},E_{\bar{k}})$ is the Klein group. 
	
	If $p\in P(k')$ is a $k'$-rational point, denote by $\bar{p}$ its Galois conjugate. If $n=4m-2$, choose $p\in E(k')$, otherwise $p\in P\setminus E(k')$. If $n=8$, for every non-trivial $g\in\aut(P_{\bar{k}},E_{\bar{k}})$ the set of points $p\in E(k')$ such that either $g(p)=p$ or $g(p)=\bar{p}$ is not Zariski-dense, since $k$ is infinite we may choose $p\in P\setminus E(k')$ so that $g(p)\neq p$, $g(p)\neq\bar{p}$ for every non-trivial $g\in\aut(P_{\bar{k}},E_{\bar{k}})$.
	
	Consider a cyclic cover $f:\PP^{1}_{\bar{k}}\to P_{\bar{k}}$ of degree $2$ which ramifies at $p$ and $\bar{p}$ and let $D\subset \PP^{1}_{\bar{k}}$ be the inverse image of $E_{\bar{k}}$ with the reduced structure, it is an effective divisor of degree $n$. The group of automorphisms $\aut(\PP^{1},D)$ has a cyclic subgroup $\aut(f)$ of order $2$. Let $C\subset\aut(\PP^{1},D)$ be the centralizer of $\aut(f)$, then $C/\aut(f)$ acts faithfully on $(P_{\bar{k}},E_{\bar{k}})$. If $n\ge 10$, $\aut(P_{\bar{k}},E_{\bar{k}})$ is trivial and hence $C=\aut(f)$. If $n=8$ and $g\in C/\aut(f)\subset \aut(P_{\bar{k}},E_{\bar{k}})$, then $g(p)$ is a branch point for $f$ and hence either $g(p)=p$ or $g(p)=\bar{p}$. By our choice of $p$, we get that $g$ is trivial, hence $C=\aut(f)$ in this case, too.
\end{proof}

\begin{lemma}\label{lem:grp}
    Let $G$ be a finite group and $g\in G$ a non trivial element with $g^2=\id$ and such that the centralizer of $g$ is $<g>$. Then $|G|/2$ is odd.
\end{lemma}

\begin{proof}
    Let $P\subset G$ be a $2$-Sylow subgroup containing $g$, the center $Z$ of $P$ is non-trivial. Since the centralizer of $g$ is $<g>$, we get that $Z\subset <g>$ and hence $Z=<g>$. It follows that $P$ centralizes $g$, hence $P=Z=<g>$ and $|G|/2=[G:P]$ is odd. 
\end{proof}

\begin{proposition}\label{prop:deg8}
	Let $k$ be a field of characteristic $\neq 2$, assume that the $2$-torsion $\on{Br}(k)[2]$ of the Brauer group is non-trivial. Let $n\ge 8$ be an even integer. There exists an effective, reduced divisor $D\subset \PP^{1}_{\bar{k}}$ of degree $n$ with field of moduli equal to $k$ such that $(\PP^{1}_{\bar{k}},D)$ is not defined over $k$.
\end{proposition}

\begin{proof}
	Since the $2$-torsion of the Brauer group is non trivial, in particular $k$ is infinite, and there exists a non-trivial rational curve $P$. Let $E\subset P$, $p\in E(k')$, $f:\PP^{1}_{\bar{k}}\to P_{\bar{k}}$, $D\subset \PP^{1}_{\bar{k}}$ as in Lemma~\ref{lem:prep}. 
	
	Let $k(D)/k$ be the field of moduli of $(\PP^{1}_{\bar{k}},D)$ and $\cG$ the residue gerbe. Consider $p$ as a divisor of degree $2$, let $\sqrt[2]{P,p}$ be $2$nd root stack and $\sqrt[2]{P,p}\to\Pi$ its étale fundamental gerbe. We have that $\Pi$ is an abelian gerbe banded by $\mu_{2}$ and $\sqrt[2]{P,p}\to\Pi$ is representable, it can be checked after base changing to $\bar{k}$.
	
	 Let $E'\subset \sqrt[2]{P,p}$ be the inverse image of $E\subset P$ with the reduced structure, since $\sqrt[2]{P,p}\to\Pi$ is representable then $E'$ finite étale over $\Pi$ and the geometric fibers of $(\sqrt[2]{P,p},E')\to\Pi$ are isomorphic to $(\PP^{1}_{\bar{k}},D)$. Let $\cG$, $\cP\to\cG$ be the residual gerbe and the universal family of $(\PP^{1}_{\bar{k}},D)$ respectively. By definition of residual gerbe, we have a $2$-cartesian diagram
	\[\begin{tikzcd}
		\sqrt[2]{P,p}\rar\dar	&	\cP\dar	\\
		\Pi\rar					&	\cG.
	\end{tikzcd}\]
	Since $\Pi$ is a gerbe over $k$ and there is a factorization $\Pi\to\cG\to\spec k(D)\to\spec k$, we get that $k(D)=k$.
	
	By Lemma~\ref{lem:grp}, $\aut(f)$ has odd index in $\aut(\PP^{1},D)$. Let $2a$ with $a$ odd be the order of $\aut(\PP^{1},D)$. Let $\bP$ be the compression of $(\PP^{1}_{\bar{k}},D)$, by the diagram above there is a natural induced map $P\to \bP$ of degree $a$. Since $a$ is odd and $P$ is a non-trivial rational curve, it follows that $\bP$ is a non-trivial rational curve, since otherwise $P$ would have a divisor of odd degree. By Corollary~\ref{cor:cyclic}, we get that $\aut(\PP^{1},D)$ is cyclic of even order and hence $(\PP^{1}_{\bar{k}},D)$ is not defined over $k$ by Theorem~\ref{thm:divisors}.
\end{proof}

\section{Generalizations of results of A. Marinatto and B. Huggins}

The following lemma connects our results to those of Marinatto.

\begin{lemma}\label{lem:compoint}
	Let $k$ be a field and $D\subset\PP^{1}_{\bar{k}}$ an effective, reduced divisor of degree $n\ge 3$ on $\PP^{1}_{\bar{k}}$ with $n\ge 3$. The following are equivalent.
	\begin{itemize}
		\item $D$ descends to a divisor of $\PP^{1}$ over the field of moduli.
		\item The universal family $\cP\to\cG$ has a rational point.
		\item The compression $\bP$ has a rational point.
	\end{itemize}
\end{lemma}

\begin{proof}
	By base change, we may assume that $k$ is the field of moduli. If $D$ descends to a divisor of $\PP^{1}$ over $k$, then by definition there is a morphism $\PP^{1}_{k}\to\cP$, hence the universal family $\cP$ has a rational point, which in turn implies that the compression $\bP$ has a rational point. Now assume that $\bP(k)$ is non-empty, we want to show that $D$ descends to a divisor of $\PP^{1}$ over $k$.
	
	Assume first that $k$ is finite. By \cite[Theorem 8.1]{diproietto-tonini-zhang}, $\cG(k)$ is non-empty, hence there exists a model $(P,D_{0})$ of $(\PP^{1}_{\bar{k}},D)$ over $k$. Since $k$ is finite, then $P=\PP^{1}_{k}$.
	
	Assume now that $k$ is infinite. As shown in the proof of \cite[Theorem 5.4]{giulio-angelo-moduli}, there exists an open subset $U\subset\cP$ which is a scheme. Since $k$ is infinite and $U(k)\subset\bP(k)$ is open, then $U(k)$ is non-empty, hence there exists a rational point $p\in \cP(k)$. The image of $p$ in $\cG(k)$ gives, by definition of $\cG$ and $\cP$, a model $(P,D_{0})$ of $(\PP^{1}_{\bar{k}},D)$ over $k$ such that $P(k)\neq\emptyset$.
\end{proof}

\begin{theorem}[{\cite[Theorem 1]{marinatto}}]\label{thm:marinatto1}
	Let $k$ be a field and $D\subset\PP_{\bar{k}}^{1}$ a reduced, effective divisor of $\PP^{1}_{\bar{k}}$ of degree $n\ge 3$. If $n$ is odd, then $D$ descends to a divisor of $\PP^{1}$ over the field of moduli. 
\end{theorem}

\begin{proof}
	Thanks to \cite[Proposition 4.1]{giulio-angelo-moduli} and Corollary~\ref{cor:cha2}, we may assume that $k$ is infinite with $\cha k\neq 2$. By base change, we can also assume that $k$ is the field of moduli of $(\PP_{\bar{k}}^{1},D)$. Let $\bP$ be the compression and $\bD\subset\bP$ the compressed divisor, we have that $\bD$ is a finite étale scheme over $k$.
	
	By Lemma~\ref{lem:compoint} it is enough to show that, if $\bP(k)=\emptyset$, then $n$ is even. If $p,q\in \bD(k^{s})=\bD(\bar{k})$ are in the same Galois orbit, then the fibers of $D\to \bD_{\bar{k}}=D/\aut(\PP^{1},D)$ over $p$ and $q$ have the same cardinality. Since $\bP$ is non-trivial and $\bD$ is étale over $k$, the Galois-orbits of $\bD(k^{s})$ have even cardinality. It follows that $D$ has even cardinality, too.
\end{proof}

\begin{proposition}[{\cite[Proposition 2.12]{marinatto}}]\label{prop:marinatto2}
	Let $k$ be a field and $D\subset\PP^{1}_{\bar{k}}$ an effective, reduced divisor of degree $4$. Then $D$ descends to a divisor of $\PP^{1}$ over the field of moduli.
\end{proposition}

\begin{proof}
	For every such divisor, there is a copy of $\ZZ/2\times\ZZ/2$ inside $\aut(\PP^{1}_{\bar{k}},D)$, hence the compression has a rational point by Proposition~\ref{prop:cyclicgerbe}. We conclude by Lemma~\ref{lem:compoint}.
\end{proof}

\begin{theorem}[{\cite[Theorem 5.3]{huggins}}]
	Let $k$ be a field of characteristic $\neq 2$ and $H$ a hyperelliptic curve over $\bar{k}$ with hyperelliptic involution $\iota$. If $H$ is not defined over its field of moduli, then $\underaut(H)/<\iota>$ is cyclic of order prime to $\operatorname{char}k$.
\end{theorem}

\begin{proof}
	Let $D\subset\PP^{1}_{\bar{k}}=H/\iota$ be the branch divisor, we have $\aut(X)/<\iota>=\aut(\PP^{1}_{\bar{k}},D)$. Let $\cG_{H},\cP_{H},\cG_{D},\cP_{D}$ be the residue gerbes and the universal families of $H,D$ respectively cf. \cite[\S 3, \S 5]{giulio-angelo-moduli}, there is a cartesian diagram
	\[\begin{tikzcd}
		\cP_{H}\rar\dar		&	\cP_{D}\dar	\\
		\cG_{H}\rar			&	\cG_{D}
	\end{tikzcd}\]
	where the morphism $\cG_{H}\to\cG_{D}$ is an abelian gerbe banded by $\mu_{2}$. It follows that $(\PP^{1}_{\bar{k}},D)$ and $H$ have equal fields of moduli and equal compression $\bP$. Thanks to \cite[Proposition 4.1]{giulio-angelo-moduli}, we may assume that $k$ is infinite. Since $H$ is not defined over the field of moduli, i.e. $\cG_{H}(k)=\emptyset$, and there is a rational map $\bP\dashrightarrow\cG_{H}$ \cite[Proof of Theorem 5.4]{giulio-angelo-moduli}, we get that $\bP(k)=\emptyset$. We conclude by applying Corollary~\ref{cor:cyclic} to $\bP\dashrightarrow\cG_{D}$.
\end{proof}

\bibliographystyle{amsalpha}
\bibliography{main}

\end{document}